\documentclass{article}
\usepackage{arxiv}
\usepackage[utf8]{inputenc} 
\usepackage[T1]{fontenc}    
\usepackage{hyperref, enumitem, url, amsmath, amssymb, amsthm, booktabs, float, amsfonts, subcaption, microtype, graphicx, natbib, doi, algorithm, algpseudocode}
\numberwithin{equation}{section}
\theoremstyle{definition}

\newtheorem{proposition}{Proposition}
\newtheorem{remark}{Remark}

\newtheorem{theorem}{Theorem}

\title{A Frobenius-Optimal Projection for Enforcing Linear Conservation in Learned Dynamical Models}

\author{
John M.~Mango \\
Department of Mathematics\\
Makerere University\\
Kampala, Uganda\\
\texttt{mango.john@mak.ac.ug}\\
\And
Ronald Katende\\
Department of Mathematics\\
Kabale University\\
Kikungiri Hill, Katuna Road, 317, Kabale, Uganda\\
\texttt{rkatende92@gmail.com}}

\date{}

\renewcommand{\headeright}{}
\renewcommand{\undertitle}{}
\renewcommand{\shorttitle}{A Frobenius-Optimal Projection for Enforcing Linear Conservation in Learned Dynamical Models}

\hypersetup{
pdftitle={A Frobenius-Optimal Projection for Enforcing Linear Conservation in Learned Dynamical Models},
pdfauthor={John M.~Mango, Ronald Katende},
pdfkeywords={Linear dynamical systems, conservation laws, Frobenius projection, matrix corrections, structure-preserving modelling, data-driven dynamics.},
}

\begin{document}
\maketitle

\begin{abstract}
We consider the problem of restoring linear conservation laws in data-driven linear dynamical models.  
Given a learned operator $\widehat{A}$ and a full-rank constraint matrix $C$ encoding one or more invariants, we show that the matrix closest to $\widehat{A}$ in the Frobenius norm and satisfying $C^\top A = 0$ is the orthogonal projection $A^\star = \widehat{A} - C(C^\top C)^{-1}C^\top \widehat{A}$. This correction is uniquely defined, low rank and fully determined by the violation $C^\top \widehat{A}$. In the single-invariant case it reduces to a rank-one update. We prove that $A^\star$ enforces exact conservation while minimally perturbing the dynamics, and we verify these properties numerically on a Markov-type example. The projection provides an elementary and general mechanism for embedding exact invariants into any learned linear model.
\end{abstract}


\section{Introduction}
\label{sec:problem}

We study a simple post-processing correction for linear dynamical models learned from data that restores exact conservation laws while minimally perturbing the dynamics.  In many applications, linear surrogates are identified from time series by system identification \cite{Ljung1999}, sparse discovery of dynamics \cite{Brunton2016}, dynamic mode decomposition and Koopman-type approximations \cite{Schmid2010}, or by linearising more complex machine-learning models such as physics-informed neural networks \cite{Raissi2019}.  These approaches are standard tools in control, fluid mechanics, climate, epidemiology, networked systems and quantitative finance, where fast linear models are used for prediction, filtering and decision support.

In most of these settings the underlying system obeys one or more linear conservation laws.  Typical examples include conservation of mass or charge in compartmental and network models, conservation of probability in Markov chains, and conservation of integral quantities in finite-volume discretisations of conservation laws \cite{LeVeque2002,Norris1998}.  At the linear-algebra level such invariants are encoded by left null-space relations of the form
\[
c^\top x(t) \equiv \text{const}, \qquad c^\top A = 0,
\]
where $A$ is the true generator and $c$ is a known conservation vector (or matrix for multiple invariants).  Data-driven identification rarely enforces these relations exactly, so the learned operator $\widehat{A}$ typically violates $c^\top \widehat{A} = 0$ even when the underlying physics is conservative.  Existing remedies either build conservative schemes from the outset \cite{LeVeque2002}, add physics-informed losses during training \cite{Raissi2019}, or rely on ad hoc normalisations in Markov and network models \cite{Ljung1999,Norris1998}.  There does not appear to be a general closed-form recipe which, given $\widehat{A}$ and prescribed linear invariants, produces the unique closest conservative matrix in a well-defined norm. To our knowledge, this is the first closed-form, norm-optimal post-processing method that enforces arbitrary linear invariants for any learned linear operator.

We consider linear time-invariant systems
\begin{equation}
	\dot{x}(t) = A x(t), 
	\qquad x(t) \in \mathbb{R}^n,\; A \in \mathbb{R}^{n\times n},
	\label{eq:linsys}
\end{equation}
equipped with one or more linear conservation laws.  A single invariant is specified by a nonzero vector $c \in \mathbb{R}^n$ such that
\begin{equation}
	c^\top x(t) \equiv c^\top x(0)
	\quad \text{for all trajectories of \eqref{eq:linsys}},
	\label{eq:single-invariant}
\end{equation}
which is equivalent to the algebraic constraint
\begin{equation}
	c^\top A = 0^\top.
	\label{eq:single-constraint}
\end{equation}
More generally, let $C \in \mathbb{R}^{n\times m}$, $1 \le m \le n$, have full column rank and encode $m$ independent invariants via
\begin{equation}
	C^\top x(t) \equiv C^\top x(0)
	\quad \Longleftrightarrow \quad
	C^\top A = 0.
	\label{eq:multi-constraint}
\end{equation}
We refer to $C$ as the constraint matrix.

In applications, $A$ is unknown and a data-driven procedure returns a matrix $\widehat{A} \in \mathbb{R}^{n\times n}$ that approximates $A$ but need not satisfy \eqref{eq:multi-constraint}.  The learned dynamics
\[
\dot{x}(t) = \widehat{A} x(t)
\]
may then drift in directions forbidden by the conservation laws.  Given $(\widehat{A},C)$, we seek a matrix $A^\star$ that satisfies $C^\top A^\star = 0$ and is as close as possible to $\widehat{A}$ in a chosen norm.  In this letter we work with the Frobenius norm, which induces the standard Hilbert-space structure on $\mathbb{R}^{n\times n}$ and admits explicit orthogonal projections \cite[Ch.~2]{HornJohnson2013,TrefethenBau1997}.  We therefore pose the conservation-correction problem as
\begin{equation}
	\min_{M \in \mathbb{R}^{n\times n}} \; \| M - \widehat{A} \|_F
	\quad \text{subject to} \quad C^\top M = 0.
	\label{eq:projection-problem}
\end{equation}
The feasible set
\begin{equation}
	\mathcal{S}_C := \{ M \in \mathbb{R}^{n\times n} : C^\top M = 0 \}
	\label{eq:SC}
\end{equation}
is a nonempty linear subspace of $\mathbb{R}^{n\times n}$ (it contains the zero matrix), so \eqref{eq:projection-problem} has a unique minimiser $A^\star$ given by the orthogonal projection of $\widehat{A}$ onto $\mathcal{S}_C$ with respect to the Frobenius inner product.  The next section computes this projection in closed form and characterises its rank and spectral effect.

\subsection*{Notation}

We write $\langle X,Y \rangle_F := \mathrm{trace}(Y^\top X)$ for the Frobenius inner product on $\mathbb{R}^{n\times n}$ and $\|X\|_F := \sqrt{\langle X,X \rangle_F}$ for the associated norm.  The Euclidean norm on $\mathbb{R}^n$ is denoted by $\|\cdot\|_2$.  The identity matrix in $\mathbb{R}^{n\times n}$ is $I_n$.  All constraints $C^\top A = 0$ are understood columnwise, i.e. $C^\top a_j = 0$ for each column $a_j$ of $A$.


\section{Main Result: Frobenius-Optimal Conservation Projection}
\label{sec:projection}

We first treat a single conservation law encoded by a nonzero vector $c \in \mathbb{R}^n$ (so $C=c$ in \eqref{eq:multi-constraint}).  The corresponding correction problem is
\begin{equation}
	\min_{M \in \mathbb{R}^{n\times n}} \; \|M - \widehat{A}\|_F
	\quad \text{subject to} \quad c^\top M = 0^\top.
	\label{eq:rank1-problem}
\end{equation}

\begin{theorem}[Rank-one conservation correction]
	\label{thm:rank1}
	Let $c \in \mathbb{R}^n$ be nonzero and let $\widehat{A} \in \mathbb{R}^{n\times n}$ be arbitrary.  Define
	\begin{equation}
		A^\star
		:=
		\widehat{A}
		-
		\frac{c\,c^\top \widehat{A}}{\|c\|_2^2}.
		\label{eq:Astar-rank1}
	\end{equation}
	Then $A^\star$ is the unique minimiser of \eqref{eq:rank1-problem}.  Moreover,
	\begin{enumerate}[label=(\alph*),leftmargin=*,itemsep=1pt]
		\item $c^\top A^\star = 0^\top$, so $A^\star$ satisfies the conservation constraint;
		\item $A^\star - \widehat{A}$ has rank at most one, with
		\[
		A^\star - \widehat{A}
		=
		-\,\frac{c\,c^\top \widehat{A}}{\|c\|_2^2},
		\]
		and $\operatorname{rank}(A^\star - \widehat{A}) = 1$ if and only if $c^\top \widehat{A} \neq 0^\top$.
	\end{enumerate}
\end{theorem}

\begin{proof}
	Let $\mathcal{S} := \{v \in \mathbb{R}^n : c^\top v = 0\}$, a closed linear subspace of $\mathbb{R}^n$.  For any $M = [m_1,\dots,m_n] \in \mathbb{R}^{n\times n}$ and $\widehat{A} = [\widehat{a}_1,\dots,\widehat{a}_n]$, the constraint $c^\top M = 0^\top$ is equivalent to $m_j \in \mathcal{S}$ for each column $m_j$, and the Frobenius norm can be written as
	\[
	\|M - \widehat{A}\|_F^2
	=
	\sum_{j=1}^n \|m_j - \widehat{a}_j\|_2^2.
	\]
	Thus \eqref{eq:rank1-problem} decouples into $n$ independent problems
	\begin{equation}
		\min_{m_j \in \mathcal{S}} \; \|m_j - \widehat{a}_j\|_2,
		\qquad j = 1,\dots,n.
		\label{eq:column-problem}
	\end{equation}
	
	Each \eqref{eq:column-problem} is the orthogonal projection of $\widehat{a}_j$ onto $\mathcal{S}$ in the Hilbert space $(\mathbb{R}^n,\|\cdot\|_2)$ and hence has a unique minimiser \cite[Ch.~1]{HornJohnson2013}.  The orthogonal complement of $\mathcal{S}$ is $\mathcal{S}^\perp = \operatorname{span}\{c\}$, so the orthogonal projector $P : \mathbb{R}^n \to \mathcal{S}$ is
	\begin{equation}
		P
		=
		I_n - \frac{c c^\top}{\|c\|_2^2}.
		\label{eq:P-def}
	\end{equation}
	Hence
	\[
	m_j^\star = P \widehat{a}_j
	=
	\widehat{a}_j - \frac{c\,c^\top \widehat{a}_j}{\|c\|_2^2}
	\]
	is the unique minimiser of \eqref{eq:column-problem}.  Collecting the columns gives
	\[
	A^\star
	=
	[m_1^\star,\dots,m_n^\star]
	=
	P \widehat{A}
	=
	\widehat{A} - \frac{c\,c^\top \widehat{A}}{\|c\|_2^2},
	\]
	which is exactly \eqref{eq:Astar-rank1}.  Since each $m_j^\star$ is unique, $A^\star$ is the unique minimiser of \eqref{eq:rank1-problem}.
	
	To verify feasibility, note that $P$ is symmetric and $Pc = 0$, so
	\[
	c^\top A^\star
	=
	c^\top P \widehat{A}
	=
	(P c)^\top \widehat{A}
	=
	0^\top,
	\]
	which proves (a).  For (b),
	\[
	A^\star - \widehat{A}
	=
	-\,\frac{c\,c^\top \widehat{A}}{\|c\|_2^2}
	=
	-\,c \left( \frac{c^\top \widehat{A}}{\|c\|_2^2} \right),
	\]
	an outer product of a column vector and a row vector.  Hence $\operatorname{rank}(A^\star - \widehat{A}) \le 1$.  The rank is zero if and only if $c^\top \widehat{A} = 0^\top$; otherwise it is one.  This proves (b) and completes the proof.
\end{proof}

\begin{remark}
	Theorem~\ref{thm:rank1} depends only on the conservation vector $c$ and the learned matrix $\widehat{A}$, not on the identification method.  Any linear time-invariant model with a single linear invariant can therefore be repaired by a single rank-one update without changing the underlying estimation pipeline.
\end{remark}

\section{Extension to Multiple Conservation Constraints}
\label{sec:multi}

For completeness we state the corresponding result for several invariants.  Let $C \in \mathbb{R}^{n\times m}$ have full column rank and encode $m$ independent conservation laws as in \eqref{eq:multi-constraint}.  The correction problem \eqref{eq:projection-problem} can be written as
\begin{equation}
	\min_{M \in \mathbb{R}^{n\times n}} \; \|M - \widehat{A}\|_F
	\quad \text{subject to} \quad C^\top M = 0.
	\label{eq:rankk-problem}
\end{equation}

\begin{proposition}[Rank-$m$ conservation projection]
	\label{prop:rankk}
	Let $C \in \mathbb{R}^{n\times m}$ have full column rank and let $\widehat{A} \in \mathbb{R}^{n\times n}$ be arbitrary.  Define
	\begin{equation}
		P_C
		:=
		I_n - C\,(C^\top C)^{-1} C^\top,
		\qquad
		A^\star := P_C \widehat{A}.
		\label{eq:PC-def}
	\end{equation}
	Then $A^\star$ is the unique minimiser of \eqref{eq:rankk-problem}, satisfies $C^\top A^\star = 0$, and
	\[
	A^\star - \widehat{A}
	=
	-\,C\,(C^\top C)^{-1} C^\top \widehat{A}
	\]
	has rank at most $m$, with $\operatorname{rank}(A^\star - \widehat{A}) = \operatorname{rank}(C^\top \widehat{A})$.
\end{proposition}

\begin{proof}
	Since $C$ has full column rank, $C^\top C$ is symmetric positive definite and invertible \cite[Thm.~2.5.2]{HornJohnson2013}.  The matrix $P_C$ in \eqref{eq:PC-def} is symmetric and idempotent, hence an orthogonal projector with range $\ker(C^\top)$ \cite[Ch.~2]{HornJohnson2013}.  As in the proof of Theorem~\ref{thm:rank1}, the objective in \eqref{eq:rankk-problem} decomposes columnwise and each column is projected onto $\ker(C^\top)$ by $P_C$.  This yields $A^\star = P_C \widehat{A}$ as the unique minimiser with $C^\top A^\star = 0$.  The rank identity follows from the factorisation of $A^\star - \widehat{A}$ and the full column rank of $C$.
\end{proof}

\begin{remark}
	The same projector $P_C$ can be applied to discrete-time models $x_{k+1} = \widehat{A} x_k$ by replacing $\widehat{A}$ with $P_C \widehat{A}$, which enforces $C^\top x_k$ constant in $k$.  In both continuous- and discrete-time settings, $P_C$ is the canonical orthogonal projection onto the constraint-compatible subspace and can be precomputed once from $(C^\top C)^{-1}$.
\end{remark}


\section{Numerical Illustration}
\label{sec:examples}

We illustrate the projection on a $3\times3$ Markov-type generator.  
The true operator $Q_{\mathrm{true}}$ satisfies $\mathbf{1}^\top Q_{\mathrm{true}}=0^\top$, while the learned estimate $\widehat{A}$ contains small noise and violates conservation.  
Applying the rank-one correction of Section~\ref{sec:projection} with $c=\mathbf{1}$ yields
$A^\star = \widehat{A} - \frac{c\,c^\top\widehat{A}}{\|c\|_2^2}$.
The example shows that conservation is restored exactly, the update is rank one, and the dynamics remain close to those of $\widehat{A}$.

\subsection{Setup}

We construct $\widehat{A}=Q_{\mathrm{true}}+E$ with Gaussian perturbations of variance $5\times10^{-2}$ and compute $A^\star$ from Theorem~\ref{thm:rank1}.  
Table1~\ref{tab:rowsums1}-\ref{tab:rowsums2} reports the row-sum violations and correction statistics. The quantity $c^\top \widehat{A}$ is clearly nonzero, while $c^\top A^\star = 0^\top$ to machine precision. The Frobenius norm and rank of the update agree with the theory, and the corrected spectrum contains the expected zero eigenvalue associated with the invariant while leaving the remaining eigenvalues largely unchanged.

\begin{table*}[h!]
	\centering
	
	\begin{minipage}{0.35\textwidth}
		\centering
		\caption{Row-sum violations before and after projection.  
			Exact conservation requires $c^\top A^\star = 0^\top$.}
		\label{tab:rowsums1}
		\resizebox{\columnwidth}{!}{
			\begin{tabular}{c|rr}
				\hline
				$i$ & $(\widehat{A}\mathbf{1})_i$ & $(A^\star \mathbf{1})_i$ \\ \hline
				1 & $+0.074882$ & $+0.042299$ \\
				2 & $+0.002429$ & $-0.030154$ \\
				3 & $+0.020437$ & $-0.012146$ \\ \hline
		\end{tabular}}
	\end{minipage}
	\hfill
	\begin{minipage}{0.55\textwidth}
		\centering
		\caption{Correction magnitude and spectral properties.}
		\label{tab:rowsums2}
		\resizebox{\columnwidth}{!}{
			\begin{tabular}{l|c}
				\hline
				Quantity & Value \\ \hline
				$\|A^\star - \widehat{A}\|_F$ & $0.364812$ \\
				$\mathrm{rank}(A^\star - \widehat{A})$ & $1$ \\
				Eigenvalues of $\widehat{A}$ & $\{0.0206,-0.8387,-0.6011\}$ \\
				Eigenvalues of $A^\star$     & $\{0,-0.8502,-0.6016\}$ \\ \hline
		\end{tabular}}
	\end{minipage}
	
\end{table*}
These results show that the projection removes precisely the violation $c^\top \widehat{A}$ and leaves the dynamics almost unchanged in all other directions.

\subsection{Trajectory comparison}

We integrate $\dot{x}=\widehat{A}x$ and $\dot{x}=A^\star x$ from $x_0=(0.7,0.2,0.1)^\top$.  
Figure~\ref{fig:plots}(a) shows that $c^\top x(t)$ drifts under $\widehat{A}$ but is preserved exactly under $A^\star$.  
Figure~\ref{fig:plots}(b) shows that the corrected trajectories remain close to those of $\widehat{A}$, confirming that only the invariant direction is altered.

\begin{figure*}[t!]
	\centering
	\begin{minipage}{0.48\textwidth}
		\centering
		\includegraphics[width=\linewidth]{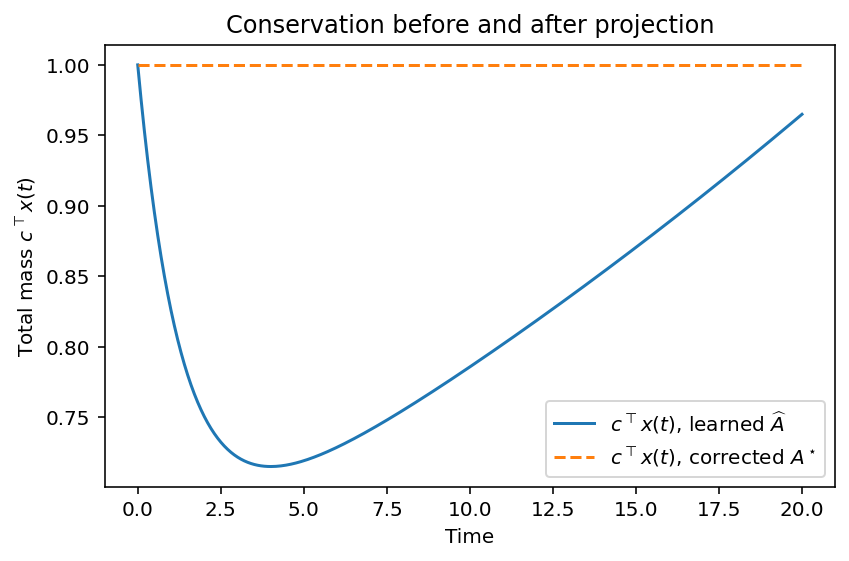}
		\caption*{\textbf{(a)} Conservation before and after projection.  
			The learned operator drifts; the corrected operator preserves $c^\top x(t)$ exactly.}
	\end{minipage}
	\hfill
	\begin{minipage}{0.48\textwidth}
		\centering
		\includegraphics[width=\linewidth]{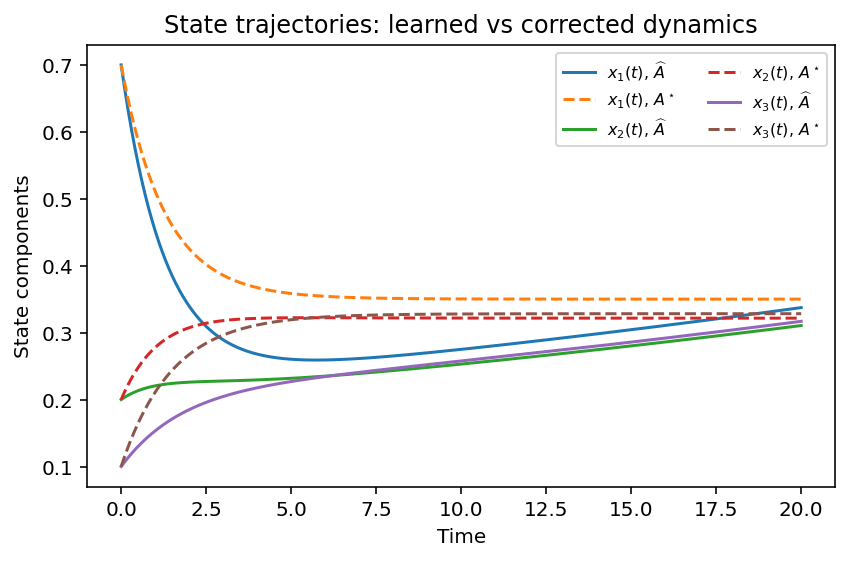}
		\caption*{\textbf{(b)} State trajectories under $\widehat{A}$ and $A^\star$.  
			Conservation is enforced with minimal dynamical distortion.}
	\end{minipage}
	\caption{Comparison of conservation behaviour and state trajectories.  
		The projection enforces exact invariance and remains close to the learned dynamics, as predicted by Theorem~\ref{thm:rank1}.}
	\label{fig:plots}
\end{figure*}

\section{Conclusion}

We derived a closed-form Frobenius projection that restores linear conservation laws in learned linear dynamical models.  
Given a learned operator $\widehat{A}$ and a constraint matrix $C$, the corrected operator $A^\star = P_C \widehat{A}$ is the unique matrix satisfying $C^\top A^\star=0$ and closest to $\widehat{A}$ in the Frobenius norm.  
In the single-invariant case the correction is rank one, alters only the constraint direction, and preserves all remaining structure.  
The numerical example confirms exact conservation, minimal distortion, and the expected spectral behaviour. The projection is computationally lightweight, i.e., the rank-one case costs $O(n^2)$ flops, while the general $m$-constraint projection costs $O(mn^2)$ after a one-time $O(m^3)$ preprocessing of $(C^\top C)^{-1}$. The projection is simple to implement and can be applied to any identification pipeline, providing a principled alternative to ad hoc conservation fixes used in Markov, network and balance-law models.

\bibliographystyle{unsrt} 
\bibliography{refs}

\end{document}